\documentclass[10pt,leqno]{amsart}
\usepackage{graphicx}
\usepackage{indentfirst,csquotes}

\usepackage{amssymb,amsthm,amsmath}
\numberwithin{equation}{section}
\usepackage{xcolor,paralist,hyperref,titlesec,fancyhdr,etoolbox}
\newtheorem{thm}{Theorem}[section]

\newtheorem{lem}{Lemma}[section]
\newtheorem{prop}{Proposition}[section]

\titleformat{\section}[hang]{\normalfont\large\bfseries}{\thesection. }{0pt}{}
\titlespacing*{\section}{0pt}{2ex}{2ex}

\titleformat{\subsection}[hang]{\normalsize\bfseries}{\thesubsection. }{0pt}{}
\titlespacing*{\section}{0pt}{2ex}{2ex}

\hypersetup{ colorlinks=true, linkcolor=black, filecolor=black, urlcolor=black }

\usepackage{lipsum}

\begin{document}
\title{Formation of singularities for the relativistic membrane equation with radial symmetry} 

\author[author1]{Lv Cai}
\author[author2]{Jianli Liu}
\date{\today}
\address{Department of Mathematics, Shanghai University, 99 Shangda Road, Shanghai, 200444, China}
\email{cailv@shu.edu.cn}
\email{jlliu@shu.edu.cn}
\maketitle

\let\thefootnote\relax

\begin{abstract}
The relativistic membrane equation can be rewritten as a first order hyperbolic system. 
Making use of the characteristic decomposition method, a new blow-up theorem is established. As an application, it demonstrates the formation of singularities for the relativistic membrane equation. Indeed, the singularity occurs when the hypersurface turns from being timelike to being null. This generalizes the result of Kong, Sun and Zhou's work for one-dimensional case [J Math Phys 47(1): 013503, 2006].

\end{abstract} 




\bigskip

\section{Introduction}
\par

The study of singularity for nonlinear wave equations has long been a central topic in geometric analysis and mathematical physics. Among these, the relativistic membrane equation occupies a particularly significant position as it describes the dynamics of timelike extremal hypersurfaces in the Minkowski spacetime $\mathbb{R}^{1+(n+1)} (n \geq 2)$.

Mathematically, the relativistic membrane equation can be derived as the Euler-Lagrange equation for the volume functional of timelike hypersurfaces.
Let $(t,x_{1},\cdots,x_{n+1})$ be coordinates in the $(1+(n+1))$-dimentional Minkowski spacetime endowed with the metric
\begin{equation}
	\mathrm{d}s^2 = - \mathrm{d} t^2 + \sum\limits_{i=1}^{n+1} \mathrm{d}x_{i}^2.
\end{equation}
We consider the hypersurface $S$ denoted by a graph
\begin{equation} \label{graph}
	x_{n+1}=\phi(t,x_{1},x_{2}, \cdots, x_{n}).
\end{equation} 
Then the induced metric on the hypersurface $S$ becomes
\begin{equation} \label{induced metric}
	\begin{aligned}
		\mathrm{d}s^2 &= - \mathrm{d}t^2 + \sum\limits_{i=1}^{n} \mathrm{d}x_{i}^2 + \mathrm{d}\phi^2 \\
		& = -( 1-|\partial_{t} \phi|^2 ) \mathrm{d}t^2 + \sum\limits_{i=1}^{n} (1+|\partial_{i} \phi|^{2})\mathrm{d}x_{i}^2 + 2  \sum\limits_{i=1}^{n} \partial_{t} \phi \partial_{i} \phi \mathrm{d}t \mathrm{d}x_{i} \\
		& \ \ \ \ + 2 \sum\limits_{i,j=1,i \neq j}^{i,j=n} \partial_{i} \phi\partial_{j} \phi \mathrm{d}x_{i}\mathrm{d}x_{j} \\
		& \triangleq ( \mathrm{d}t, \mathrm{d}x_{1}, \cdots \mathrm{d}x_{n} ) G ( \mathrm{d}t, \mathrm{d}x_{1}, \cdots \mathrm{d}x_{n} )^\mathrm{T} ,
	\end{aligned}
\end{equation}
where 
\begin{equation*} 
	G = \left( 
	\begin{array}{cccc}
		-( 1-|\partial_{t} \phi|^2 ) & \partial_{t} \phi \partial_{1} \phi  & \cdots & \partial_{t} \phi \partial_{n} \phi \\
		\partial_{t} \phi \partial_{1} \phi& 1+ |\partial_{1} \phi|^2  & \cdots & \partial_{1} \phi\partial_{n} \phi \\
		\vdots & \vdots  & \ddots & \vdots\\
		\partial_{t} \phi \partial_{n} \phi& \partial_{1} \phi \partial_{n} \phi  & \cdots & 1+ |\partial_{n} \phi|^2
	\end{array}
	\right).
\end{equation*}
We assume that this hypersurface is timelike, i.e., the induced metric is Lorentzian
\begin{equation} \label{time like condition}
	1+|\nabla_{x}\phi|^{2}-|\partial_{t} \phi|^{2} >0.
\end{equation}
Thus, the volume functional for the hypersurface is
\begin{equation} \label{vol f}
	\begin{aligned}
		\mathrm{V} (\phi) 	&= \iint_{\mathbb{R} \times \mathbb{R}^{n}} \sqrt{- \mathrm{det} G} \mathrm{d}t\mathrm{d}x_{1} \cdots \mathrm{d}x_{n} \\
		& =\iint_{\mathbb{R} \times \mathbb{R}^{n}} \sqrt{1+|\nabla_{x}\phi|^{2}-|\partial_{t} \phi|^{2}} \mathrm{d}t\mathrm{d}x_{1}\cdots \mathrm{d}x_{n}.
	\end{aligned}
\end{equation}
A relativistic membrane is a critical point of the volume functional \eqref{vol f}. Consequently, we can get
\begin{equation} \label{Euler-Lagrange}
	\frac{\partial}{\partial {t}}	\frac{\partial_{t} \phi}{\sqrt{1+|\nabla_{x}\phi|^{2}-|\partial_{t} \phi|^{2}}} - \sum\limits_{i=1}^{n} \frac{\partial}{\partial {x_{i}}}\frac{\partial_i \phi}{\sqrt{1+|\nabla_{x}\phi|^{2}-|\partial_{t} \phi|^{2}}}  = 0,
\end{equation}
where $\nabla_{x}\phi = ( \partial_{1} \phi,\partial_{2} \phi,\cdots, \partial_{n} \phi)$ and $\partial_{t} \phi,\partial_{i} \phi $ denote partial differentiation with respect to $t,x_{i}$, respectively. When $n=1$, the equation \eqref{Euler-Lagrange} is the relativistic string equation, associated with Born-Infeld theory  \cite{Brenier02,Gibbons98,Hoppe13}. Before we state our main result, let us introduce some necessary works in regard to the relativistic string and membrane.

As shown above, the geometric significance of the equation \eqref{Euler-Lagrange} is manifested through the induced metric \eqref{induced metric} on the hypersurface, where the timelike condition \eqref{time like condition} ensures the Lorentzian character of the metric. This condition is essential to the hyperbolicity of the equation \eqref{Euler-Lagrange}, see \cite{EggersHoppeHynek15,Wong11}. The system is totally linearly degenerate in the sense of Lax \cite{Lax86}, and the local well-posedness for smooth initial data was proved in \cite{AuriliaChristodoulou79}. On the other hand, \eqref{Euler-Lagrange} is a quasilinear wave equation whose nonlinearity satisfies the `double null condition' \cite{Klainerman84}, and it can be reformulated as an equation in divergence form \cite{Lindblad04}. Global existence results for small initial data were obtained in \cite{Allen06,Lindblad04} by exploiting these good structures from the perspective of the wave equation. The global existence of the initial boundary value problem of time-like extremal surface equation was proved in \cite{LiuZhou08,LiuZhou09} and the asymptotic behavior to global classical solutions was given in \cite{LiuZhou07}.

In the study of large initial data leading to a globally smooth solution, 
a useful way is to examine the stability of non-trivial solutions \cite{DonningerKS16}. Recently, Liu and Zhou \cite{LiuZhou24} gave the existence and global nonlinear stability of traveling wave solutions to the timelike extremal hypersurface equation in the Minkowski space. Another attempt to study the large-data problem for nonlinear wave equations involves the short-pulse method introduced by Christodoulou in \cite{Christodoulou09}. It turns out that this method was successfully utilized to prove the formation of black holes for the vacuum Einstein equation. Later, the ideas were adapted to the relativistic membrane equation in the two and three dimensions \cite{WangWei22}, proving the global existence of smooth solution with large data. All of these results regarding small perturbations of the trivial background solution depend on the `null structure' inherent to the equations.

For one-dimensional case, Wang and Wei \cite{WangWei23} addressed the global existence for the relativistic string within the setting that the `right travelling wave' was non-small while the `left travelling wave' had to be small enough. By the approach of characteristics, Kong and Tusji \cite{KongTsuji99} proposed sufficient and necessary conditions for global solutions to the general quasilinear system with linearly degenerate characteristics, which were applied to the relativistic string in \cite{KongZhang07}.

However, formation of singularities for the relativistic membrane equation remains less understood. Early work by Hoppe \cite{Hoppe13} explored singularities in relativistic membranes using algebraic methods. Eggers and Hoppe \cite{Eggers09} constructed swallowtail singularities and observed that solutions to the relativistic membrane equation approximate those of the eikonal equation near singularities \cite{EggersHoppeHynek15}. Therefore, global results for large data mentioned above \cite{WangWei22,WangWei23} were built on a framework that the timelike condition \eqref{time like condition} holds for all time. We refer to \cite{KongWei14,KongWeiZhang14,LaiZhu22} for singularity formation of other models with linearly degenerate characteristics.

This paper investigates the formation of singularities for the radially symmetric relativistic membrane equation with
\begin{equation}
	\phi (t,x_1, \cdots, x_n) = \phi (t,r),  
\end{equation}
where $r=\left(  \sum\limits_{i=1}^{n} x_i^2\right) ^{\frac{1}{2}}$. The volume functional \eqref{vol f} is
\begin{equation}\label{A(phi)r}
	\mathrm{V}(\phi) = \frac{2\pi^{\frac{n}{2}}}{\Gamma(\frac{n}{2}) }  \iint_{\mathbb{R} \times {\mathbb{R}^+}} r^{n-1} \sqrt{1+\phi_{r}^{2}-\phi_{t}^{2}} \mathrm{d}t \mathrm{d}r,
\end{equation}
where $\Gamma(n)=\int_{0}^{+\infty} e^{-x} x^{n-1} \mathrm{d}x $ is the Gamma function. Then the corresponding Euler-Lagrange equation \eqref{Euler-Lagrange} becomes
\begin{equation}\label{Euler-Lagrange r}
	\left(\frac{r^{n-1}\phi_{t}}{\sqrt{\Delta}} \right)_{t} - \left(\frac{r^{n-1}\phi_{r}}{\sqrt{\Delta}} \right)_{r} = 0,
\end{equation}
where $ \Delta \triangleq 1+\phi_{r}^{2}-\phi_{t}^{2} $. For $r>0$, it can be rewritten as a quasilinear wave equation
\begin{equation}\label{eq1}
	(1+\phi_{r}^{2}) \phi_{tt} - 2 \phi_{r} \phi_{t} \phi_{tr} -  (1-\phi_{t}^{2}) \phi_{rr} = \frac{(n-1)\phi_{r} \Delta}{r}.
\end{equation}
Consider the Cauchy problem with initial data
\begin{equation}\label{ID phi}
	t=0: \phi=\phi_0(r) \in H_{rad}^{s} (\mathbb{R}^n), \phi_t=\phi_1(r) \in H_{rad}^{s} (\mathbb{R}^n),
\end{equation}
where $s > \frac{n+2}{2} $. For the axially symmetric case, Wong \cite{Wong18} identified the singularity arising from the degeneration of the principal symbol of the evolution, distinct from shock-type singularities. In this work, we make progress by developing a new characteristic decomposition for the radially symmetric case, establishing precise conditions under which the timelike condition fails in finite time, as well as providing a geometric interpretation of the singularity formation process.

In this paper, we reformulate the relativistic membrane equation as a first order hyperbolic system \eqref{u v eq3}. A blow-up result of \eqref{u v eq3} was established by using the characteristic decomposition method, c.f. Theorem \ref{main thm}. As an application, we can prove that for radially symmetric initial data satisfying certain natural conditions, the solution will develop a singularity in finite time through the loss of hyperbolicity. This complements and extends Kong, Sun and Zhou's work \cite{KongSunZhou06} in one-dimensional case.

The rest of this paper is organized as follows. In Section \ref{sec pre}, we rewrite the relativistic membrane equation as a first order system. In Section \ref{sec proof}, we establish and prove a new blow-up theorem leading to the formation of singularities for the the relativistic membrane equation.



\section{Preliminaries} \label{sec pre}
Let
\begin{equation}\label{s1s2}
	\begin{aligned}
		s_{1}\triangleq\phi_{r} , \ s_{2}\triangleq-\frac{\phi_{t}}{\sqrt{\Delta}},
	\end{aligned}
\end{equation}
and define
\begin{equation}\label{uv}
	\begin{aligned}
		u\triangleq\frac{s_{1}s_{2}}{\sqrt{(1+s_{1}^{2})(1+s_{2}^{2})}}, \
		v\triangleq\sqrt{(1+s_{1}^{2})(1+s_{2}^{2})}.
	\end{aligned}
\end{equation}
The transformation maps $(\phi_{t},\phi_{r})$ to $(u,v)$ as
\begin{equation}\label{uv phi}
	\begin{aligned}
		u=-\frac{\phi_{r}\phi_{t}}{1+\phi_{r}^{2}}, \ 
		v=\frac{1+\phi_{r}^{2}}{\sqrt{\Delta}}.
	\end{aligned}
\end{equation}
We can rewrite the equation \eqref{eq1} as the following first order quasilinear system.

\begin{lem} \label{lem u v eq}
	The relativistic membrane equation \eqref{eq1} can be reduced to the following system
	\begin{equation} \label{u v eq2}
		\begin{aligned}
			\begin{split}
				\left\lbrace
				\begin{array}{lr}
					u_{t} + uu_{r} + v^{-3}v_{r} = - (n-1) r^{-1} v^{-2} \mathcal{F} (u,v) &\\ 
					v_{t} + v u_{r} + uv_{r} =- (n-1) r^{-1} uv, &
				\end{array}	
				\right.
			\end{split}	
		\end{aligned}
	\end{equation}
	where the nonlinear term
	\begin{equation} \label{def of F}
		\begin{aligned}
			\mathcal{F}(u,v) \triangleq  \frac{1}{2}[ ( v^{2}-u^{2}v^{2}-1 ) - \sqrt{(v^{2}-u^{2}v^{2}-1)^2 -4u^{2}v^{2} }  ].
		\end{aligned}
	\end{equation}

\end{lem}

\begin{proof}
	By the direct calculation, one has
	\begin{equation*}
		\begin{aligned}
			u_{t} &= (-\phi_{tr} \phi_{t}-\phi_{tt} \phi_{r}) (1+\phi_{r}^{2})^{-1} +2 \phi_{tr} \phi_{t} \phi_{r}^{2} (1+\phi_{r}^{2})^{-2} \\
			&= (1+\phi_{r}^{2})^{-2} [ - \phi_{tr}\phi_{t} (1-\phi_{r}^{2}) - \phi_{tt} \phi_{r}(1+\phi_{r}^{2})], \\
			u_{r} &= (-\phi_{rr} \phi_{t}-\phi_{tr} \phi_{r}) (1+\phi_{r}^{2})^{-1} + 2 \phi_{rr} \phi_{t} \phi_{r}^{2} (1+\phi_{r}^{2})^{-2}\\
			&= (1+\phi_{r}^{2})^{-2} [ - \phi_{rr}\phi_{t} (1-\phi_{r}^{2}) - \phi_{tr} \phi_{r}(1+\phi_{r}^{2})], \\
			v_{t} &= 2 \phi_{tr} \phi_{r} \Delta^{-\frac{1}{2}} +  (1+\phi_{r}^{2}) (\Delta^{-\frac{1}{2}})_{t} \\
			&= \Delta^{-\frac{3}{2}} [ \phi_{tr} \phi_{r} ( 1+\phi_{r}^{2} - 2 \phi_{t}^{2} ) + \phi_{tt} \phi_{t} (1+\phi_{r}^{2}) ], \\
			v_{r} &= 2 \phi_{rr} \phi_{r} \Delta^{-\frac{1}{2}} +  (1+\phi_{r}^{2}) (\Delta^{-\frac{1}{2}})_{r} \\
			&= \Delta^{-\frac{3}{2}} [ \phi_{rr} \phi_{r} ( 1+\phi_{r}^{2} - 2 \phi_{t}^{2} ) + \phi_{tr} \phi_{t} (1+\phi_{r}^{2}) ]. 
		\end{aligned}
	\end{equation*}
	Then,
	\begin{equation*}
		\begin{aligned}
			&\ \ \ \ u_{t} + uu_{r} + v^{-3}v_{r} \\
			&= (-\phi_{tr} \phi_{t}-\phi_{tt} \phi_{r}) (1+\phi_{r}^{2})^{-1} + 2 \phi_{tr} \phi_{t} \phi_{r}^{2} (1+\phi_{r}^{2})^{-2} \\
			& \ \ \ \ - \phi_{t} \phi_{r} [(-\phi_{rr} \phi_{t}-\phi_{tr} \phi_{r}) (1+\phi_{r}^{2})^{-2} + 2 \phi_{rr} \phi_{t} \phi_{r}^{2} (1+\phi_{r}^{2})^{-3}] \\
			& \ \ \ \ + [ \phi_{rr} \phi_{r} ( 1+\phi_{r}^{2} - 2 \phi_{t}^{2} ) + \phi_{tr} \phi_{t} (1+\phi_{r}^{2}) ](1+\phi_{r}^{2})^{-3} \\
			&= -(1+\phi_{r}^{2})^{-2} \phi_{r} [ \phi_{tt}(1+\phi_{r}^{2}) - 2\phi_{tr} \phi_{t}\phi_{r} - \phi_{rr}(1-\phi_{t}^{2}) ].
		\end{aligned}
	\end{equation*}
	Noting the equation \eqref{eq1}, we obtain
	\begin{equation*}
		\begin{aligned}
			u_{t} + uu_{r} + v^{-3}v_{r} = -(n-1) r^{-1}  v^{-2} \phi_{r}^{2}.
		\end{aligned}
	\end{equation*}
	It follows from the representation \eqref{uv} that the new variables ($u,v$) satisfy 
	\begin{equation*}
		\begin{aligned}
			0 \leq u \leq 1, v \geq 1.
		\end{aligned}
	\end{equation*}
	Besides, by \eqref{uv phi}, if $ \phi_{r} = 0$, then $u=0$. Otherwise, if $ \phi_{r}$ does not vanish, one can derive the equation of $\phi_{r}^{2}$
	\begin{equation*}
		\begin{aligned}
			\frac{u^{2} (1+\phi_{r}^{2})^{2}}{\phi_{r}^{2}} = (1+\phi_{r}^{2}) - \frac{(1+\phi_{r}^{2})^{2}}{v^{2}},
		\end{aligned}
	\end{equation*}
	or equivalently,
	\begin{equation} \label{F eq}
		\begin{aligned}
			\phi_{r}^{4} - (v^{2}-u^{2}v^{2}-1) \phi_{r}^{2} + u^{2}v^{2}= 0.
		\end{aligned}
	\end{equation}
	We take
	\begin{equation} \label{phi_r^2 A}
		\begin{aligned}
			\phi_{r}^{2}= \frac{1}{2}[ ( v^{2}-u^{2}v^{2}-1 ) - \sqrt{(v^{2}-u^{2}v^{2}-1)^2 -4u^{2}v^{2} }  ] = \mathcal{F}(u,v).
		\end{aligned}
	\end{equation}
	Here we note that by \eqref{uv}, one has
	\begin{equation*} 
		\begin{aligned}
			& v^{2}-u^{2}v^{2}-1=s_{1}^{2}+s_{2}^{2} \geq 0, \\
			&	(v^{2}-u^{2}v^{2}-1)^2 - 4u^{2}v^{2} = (s_{1}^{2}+s_{2}^{2})^2 - 4s_{1}^{2}s_{2}^{2} \geq 0,
		\end{aligned}
	\end{equation*}
	which ensures that \eqref{phi_r^2 A} makes sense. On the other hand,
	\begin{equation*}
		\begin{aligned}
			&\ \ \ \ v_{t} + uv_{r} + u_{r}v  \\
			&= \Delta^{-\frac{3}{2}} [ \phi_{tr} \phi_{r} ( 1+\phi_{r}^{2} - 2 \phi_{t}^{2} ) + \phi_{tt} \phi_{t} (1+\phi_{r}^{2}) ] \\
			& \ \ \ \ -\Delta^{-\frac{3}{2}} \phi_{t}\phi_{r} (1+\phi_{r}^{2})^{-1} [ \phi_{rr} \phi_{r} ( 1+\phi_{r}^{2} - 2 \phi_{t}^{2} ) + \phi_{tr} \phi_{t} (1+\phi_{r}^{2}) ] \\
			& \ \ \ \ +\Delta^{-\frac{1}{2}} (1+\phi_{r}^{2})^{-1} [ \phi_{rr}\phi_{t} (-1+\phi_{r}^{2}) - \phi_{tr} \phi_{r}(1+\phi_{r}^{2}) ] \\
			&= \Delta^{-\frac{3}{2}} \phi_{t} [ \phi_{tt}(1+\phi_{r}^{2}) - 2\phi_{tr} \phi_{t}\phi_{r} - \phi_{rr}(1-\phi_{t}^{2}) ].
		\end{aligned}
	\end{equation*}
	Again by \eqref{eq1}, we obtain
	\begin{equation*}
		\begin{aligned}
			v_{t} + uv_{r} + u_{r}v = -(n-1)r^{-1}  uv.
		\end{aligned}
	\end{equation*}
	The proof of Lemma \ref{lem u v eq} is completed.

\end{proof}

If one rescales the new variables $\hat{t} = (n-1)t, \ \hat{r} = (n-1)r$ and still denotes by $t,\ r$, then \eqref{u v eq2} becomes
\begin{equation} \label{u v eq3}
	\begin{aligned}
		\begin{split}
			\left\lbrace
			\begin{array}{lr}
				u_{t} + uu_{r} + v^{-3}v_{r} = - r^{-1} v^{-2} \mathcal{F} (u,v) & \\ 
				v_{t} + v u_{r} + uv_{r} =- r^{-1} uv . &
			\end{array}	
			\right.
		\end{split}	
	\end{aligned}
\end{equation}

\section{A Blow-up Result} \label{sec proof}

The initial data we consider are
\begin{equation} \label{ID}
	\begin{aligned}
		(u,v) (r,0) = (\bar{u},v_0)(r), \ \ r \in [r_1,r_2].
	\end{aligned}
\end{equation}
The following assumptions are made throughout this paper.
\begin{equation} \label{assumption}
	\begin{aligned}
		& \text{(A1)}\ 	0 < r_1 < r_2, \ v_0 \text{ is a positive constant}, \ \bar{u} (r) \in C^1 [r_1,r_2] , \ \bar{u}(r_2) >v_{0}^{-1}; \\
		& \text{(A2)} \ \bar{u}^{\prime} (r)+\frac{\bar{u} (r)+ \beta \mathcal{F}_{0}(r) }{r} < 0, \text{ for } \ r\in [r_1,r_2], \\
		& \ \ \ \ \ \ \text{ where }1 \leq \beta<\infty \text{ is a given constant }, \mathcal{F}_{0}(r) = \mathcal{F}_{0} (\bar{u}(r),v_{0});\\
		&\text{(A3)} \ \text{there exists an interval }	(\eta_1,\eta_2) \in (r_1,r_2) 	\text{ such that } \\
		& \ \ \ \ \ \ 0 < \frac{\eta_2-\eta_1}{\bar{u} (\eta_1) - \bar{u} (\eta_2) - 2v_{0}^{-1}} \\
		& \ \ \ \ \ \ < \min \left\lbrace \frac{\eta_1-r_1}{\bar{u} (r_1) - \bar{u} (\eta_1) + 2v_{0}^{-1}},\frac{r_2 -\eta_2}{\bar{u} (\eta_2) - \bar{u} (r_2) + 2v_{0}^{-1}} \right\rbrace .
	\end{aligned}
\end{equation}
Note that if $\bar{u} (\eta_1) - \bar{u} (\eta_2) > 2v_{0}^{-1}$ and $\eta_2-\eta_1$ is sufficiently small then the assumption (A3) can be satisfied. By analyzing the dynamics along characteristic curves and applying the assumptions (A1)-(A3), we shall show that the variable
$v$ of the classical solution to system \eqref{u v eq3} with initial data \eqref{ID} blows up in finite time, leading to the loss of hyperbolicity.

\begin{thm} \label{main thm}
	There exists a $t_{*}>0$ such that system \eqref{u v eq3} with initial data \eqref{ID} admits a classical solution in $\Lambda(t_{*})$ with
	\begin{equation} \label{thm infty} 
		\begin{aligned}
			\lim\limits_{t \rightarrow t_{*}^{-}} \left\| v (\cdot,t) \right\|_{0;\Lambda(t)} = + \infty. 
		\end{aligned}
	\end{equation}	
	Here $\Lambda(t)$ is a domain defined in \eqref{def of Lambda}.
	
\end{thm}

\subsection{Characteristic decompositions}
In this subsection, we derive a group of suitable characteristic decompositions for system \eqref{u v eq3}. One can easily get the eigenvalues of system \eqref{u v eq3}
\begin{equation} \label{lambda pm}
	\begin{aligned}
		\lambda_+=u+v^{-1}, \ \lambda_-=u-v^{-1},
	\end{aligned}
\end{equation}
with the corresponding left eigenvectors $l_{\pm} = (\pm v,v^{-1})$. Multiplying \eqref{u v eq3} by $l_{\pm}$ from the left-side yields the characteristic equations
\begin{equation} \label{c rho pm}
	\begin{aligned}
		\begin{split}
			\left\lbrace
			\begin{array}{lr}	
				v \partial_{+} u + v^{-1} \partial_{+} v  = -r^{-1} (u + v^{-1}\mathcal{F}(u,v)) &\\
				- v \partial_{-} u + v^{-1} \partial_{-} v = -r^{-1} (u - v^{-1}\mathcal{F}(u,v)),&
			\end{array}	
			\right.
		\end{split}	
	\end{aligned}
\end{equation}
where 
\begin{equation} \label{def of partial pm}
	\partial_{\pm} \triangleq \partial_{t} + (u \pm v^{-1}) \partial_{r} .
\end{equation}
The equations \eqref{c rho pm} can be rewritten as
\begin{equation} \label{u c pm}
	\begin{aligned}
		\begin{split}
			\left\lbrace
			\begin{array}{lr}	
				\partial_{+} u = -v^{-2}\partial_{+} v - r^{-1}( uv^{-1}+ v^{-2}\mathcal{F}(u,v)) \\
				\partial_{-} u = v^{-2}\partial_{-} v + r^{-1}( uv^{-1}- v^{-2}\mathcal{F}(u,v)).
			\end{array}	
			\right.
		\end{split}	
	\end{aligned}
\end{equation}
Clearly, one has
\begin{equation} \label{partial rt pm}
	\begin{aligned}
		\partial_{r}=\frac{v(\partial_{+}+\partial_{-})}{2}, \ \partial_{t}=\frac{v\left[ (u+v^{-1})\partial_{-}-(u-v^{-1}) \partial_{+}\right] }{2}.
	\end{aligned}
\end{equation}
In this paper, we do not estimate the derivatives of the Riemann variants of system \eqref{u v eq3}. Instead, we will estimate $\partial_{\pm}v$. In what follows, we shall first derive some characteristic equations for $\partial_{\pm}v$. To this end, we need to introduce the following commutator relation.

\begin{prop} \label{prop1}
	One has the following commutator relation
	\begin{equation} \label{commutator prop}
		\begin{aligned}
			\partial_{+} \partial_{-} - \partial_{-} \partial_{+} = - \frac{u}{r} (\partial_{+} - \partial_{-}).
		\end{aligned}
	\end{equation}
	
\end{prop}

\begin{proof}
	By \eqref{def of partial pm}, \eqref{u c pm} and \eqref{partial rt pm}, we have
	\begin{equation*} 
		\begin{aligned}
			\partial_{+} \partial_{-} - \partial_{-} \partial_{+} & = ( \partial_{t} + \lambda_+ \partial_{r} ) ( \partial_{t} + \lambda_- \partial_{r} ) - ( \partial_{t} + \lambda_- \partial_{r} ) ( \partial_{t} + \lambda_+ \partial_{r} ) \\
			& =  \partial_{t} (\lambda_- \partial_{r}) + \lambda_+ \partial_{t} \partial_{r} + \lambda_+ \partial_{r} (\lambda_- \partial_{r}) \\
			& \ \ \ \ - \partial_{t} (\lambda_+ \partial_{r}) -\lambda_- \partial_{t} \partial_{r} - \lambda_- \partial_{r} (\lambda_+ \partial_{r}) \\
			& = ( \partial_{t}\lambda_- ) \partial_{r} + \lambda_+ ( \partial_{r}\lambda_- ) \partial_{r} - ( \partial_{t}\lambda_+ ) \partial_{r} - \lambda_- ( \partial_{r}\lambda_+ ) \partial_{r} \\
			& = ( \partial_{+}\lambda_- - \partial_{-}\lambda_+) \partial_{r} \\
			& = \frac{-2uv^{-1}}{r} \cdot \frac{v(\partial_{+} - \partial_{-})}{2} \\
			&= - \frac{u}{r} (\partial_{+} - \partial_{-}).
		\end{aligned}
	\end{equation*}
	This completes the proof of Proposition \ref{prop1}.	
\end{proof}

The next proposition states the characteristic decomposition for the variable $v$ in \eqref{u v eq3}.

\begin{prop} \label{prop2}
	One has the following characteristic decomposition about the solution $v$
	\begin{equation} \label{character decompose prop}
		\begin{aligned}
			\begin{split}
				\partial_{+} \partial_{-} v &= 2 v^{-1} \partial_{+}v \cdot \partial_{-} v + \frac{1}{2r} (2u - v^{-2} \partial_{u} \mathcal{F}) \left(  \partial_{+} v + \partial_{-} v \right)  \\
				& \ \ \ \  + \frac{1}{2r} \left( -u+v^{-1}-2v^{-1} \mathcal{F} + \partial_{v} \mathcal{F} \right)  \left(  \partial_{+} v - \partial_{-} v \right) \\
				& \ \ \ \ + \frac{1}{r^{2}}\left( 2u^2 v - uv^{-1}\partial_{u} \mathcal{F} \right) ,\\
				\partial_{-} \partial_{+} v & = 2 v^{-1} \partial_{+}v \cdot \partial_{-} v + \frac{1}{2r} (2u - v^{-2} \partial_{u} \mathcal{F}) \left(  \partial_{+} v + \partial_{-} v \right)  \\
				& \ \ \ \  + \frac{1}{2r} \left( u+v^{-1}-2v^{-1} \mathcal{F} + \partial_{v} \mathcal{F} \right)  \left(  \partial_{+} v - \partial_{-} v \right) \\
				& \ \ \ \ + \frac{1}{r^{2}}\left( 2u^2 v - uv^{-1}\partial_{u} \mathcal{F} \right)  .				 
			\end{split}	
		\end{aligned}
	\end{equation}
	
\end{prop}

\begin{proof}
	Applying the commutator relation \eqref{commutator prop} for variables $v$ and $u$ respectively, we have
	\begin{equation}\label{partial 2 - }
		\begin{aligned}
			\partial_{+} \partial_{-} v - \partial_{-} \partial_{+} v =- \frac{u}{r} \left(  \partial_{+}v - \partial_{-}v\right)   ,
		\end{aligned}
	\end{equation}
	and
	\begin{equation} \label{partial 2 u -}
		\begin{aligned}
			\partial_{+} \partial_{-} u - \partial_{-} \partial_{+} u = - \frac{u}{r} (\partial_{+}u - \partial_{-}u).
		\end{aligned}
	\end{equation}
	Inserting \eqref{u c pm} into the left side of \eqref{partial 2 u -} yields that
	\begin{equation}\label{I123}
		\begin{aligned}
			& \ \ \ \ \partial_{+} [v^{-2} \partial_{-} v + r^{-1} (uv^{-1}-v^{-2} \mathcal{F})]  - \partial_{-} [-v^{-2} \partial_{+} v + r^{-1} (uv^{-1}+v^{-2} \mathcal{F})]  \\
			& = [\partial_{+} (v^{-2}) \partial_{-} v + \partial_{-} (v^{-2}) \partial_{+} v] + v^{-2} (\partial_{+} \partial_{-} v + \partial_{-} \partial_{+} v) \\
			& \ \ \ \  + r^{-1} [ \partial_{+} (uv^{-1}) + \partial_{-} (uv^{-1}) ]  + r^{-1} [ -\partial_{+} (v^{-2}\mathcal{F}) + \partial_{-} (v^{-2}\mathcal{F}) ] \\
			& \ \ \ \  + (uv^{-1} - v^{-2} \mathcal{F})\partial_{+} (r^{-1}) + (uv^{-1} + v^{-2} \mathcal{F})\partial_{-} (r^{-1})  \\
			& \triangleq  v^{-2} [ -4 v^{-1}\partial_{+} v \cdot \partial_{-} v  + (\partial_{+} \partial_{-} v + \partial_{-} \partial_{+} v) ] + I_{1} + I_{2} + I_{3},
		\end{aligned}
	\end{equation}
	where 
	\begin{equation} \label{I1}
		\begin{aligned}
			I_{1} & \triangleq r^{-1} [ \partial_{+} (uv^{-1}) + \partial_{-} (uv^{-1}) ] \\
			& = r^{-1} [ v^{-1} (\partial_{+}u + \partial_{-}u) + u ( \partial_{+} (v^{-1}) + \partial_{-} (v^{-1}) ) ] \\
			& = r^{-1} \{v^{-1} [- v^{-2} ( \partial_{+}v-\partial_{-}v) - 2r^{-1}v^{-2}\mathcal{F}]   - u v^{-2}( \partial_{+} v + \partial_{-} v )\}     \\
			& = v^{-2} [-r^{-1} v^{-1} ( \partial_{+}v-\partial_{-}v) - r^{-1} u ( \partial_{+} v + \partial_{-} v ) - 2r^{-2}v^{-1}\mathcal{F}],
		\end{aligned}
	\end{equation}
	\begin{equation}\label{I2}
		\begin{aligned}
			I_{2} & \triangleq  r^{-1} [ -\partial_{+} (v^{-2}\mathcal{F}) + \partial_{-} (v^{-2}\mathcal{F}) ] \\
			& =  r^{-1} [ \mathcal{F} ( -\partial_{+} (v^{-2}) + \partial_{-} (v^{-2}) ) +v^{-2} (-\partial_{+} \mathcal{F} + \partial_{-} \mathcal{F}) ] \\
			& =  r^{-1} \{2v^{-3} \mathcal{F} ( \partial_{+} v - \partial_{-} v ) +v^{-2} [\partial_{u} \mathcal{F} (-\partial_{+} u + \partial_{-} u) + \partial_{v} \mathcal{F} (-\partial_{+} v + \partial_{-} v)]\}  \\
			& = r^{-1} \{2v^{-3} \mathcal{F} ( \partial_{+} v - \partial_{-} v ) +v^{-2} [ \partial_{u} \mathcal{F} (v^{-2} (\partial_{+} v + \partial_{-}v) + 2r^{-1}uv^{-1}) \\
			& \ \ \ \  - \partial_{v} \mathcal{F} (\partial_{+} v - \partial_{-} v) ]\}  \\
			& = v^{-2} [  r^{-1}v^{-2}\partial_{u} \mathcal{F} (\partial_{+} v + \partial_{-}v) + r^{-1} ( 2v^{-1}\mathcal{F} - \partial_{v} \mathcal{F} ) (\partial_{+} v - \partial_{-}v) \\
			 & \ \ \ \ + 2r^{-1}uv^{-1}\partial_{u} \mathcal{F}],
		\end{aligned}
	\end{equation}
	and
	\begin{equation} \label{I3}
		\begin{aligned}
			I_{3} & \triangleq  (uv^{-1} - v^{-2} \mathcal{F})\partial_{+} (r^{-1}) + (uv^{-1} + v^{-2} \mathcal{F})\partial_{-} (r^{-1})\\
			& = -r^{-2} (u+v^{-1}) (uv^{-1} - v^{-2} \mathcal{F}) -r^{-2} (u-v^{-1}) (uv^{-1} + v^{-2} \mathcal{F}) \\
			& =  v^{-2} \cdot r^{-2} ( -2uv + 2v^{-1}\mathcal{F}  ) .
		\end{aligned}
	\end{equation}		
	Substituting \eqref{I1}, \eqref{I2} and \eqref{I3} into \eqref{I123}, one has
	\begin{equation}
		\begin{aligned}
			& \ \ \ \ \partial_{+} [v^{-2} \partial_{-} v + r^{-1} (uv^{-1}-v^{-2} \mathcal{F})]  - \partial_{-} [-v^{-2} \partial_{+} v + r^{-1} (uv^{-1}+v^{-2} \mathcal{F})]  \\
			& =  v^{-2} [ -4 v^{-1}\partial_{+} v \cdot \partial_{-} v  + (\partial_{+} \partial_{-} v + \partial_{-} \partial_{+} v) ]\\
			& \ \ \ \ + v^{-2} r^{-1} (-u + v^{-2}\partial_{u} \mathcal{F}) ( \partial_{+} v + \partial_{-} v )  \\
			& \ \ \ \ + v^{-2}  r^{-1} ( -v^{-1} + 2v^{-1}\mathcal{F} - \partial_{v} \mathcal{F} ) (\partial_{+} v - \partial_{-}v) \\
			& \ \ \ \ + v^{-2} r^{-2} (-2uv + 2r^{-1}uv^{-1}\partial_{u} \mathcal{F}) .
		\end{aligned}
	\end{equation}
	Hence, the left-side of \eqref{partial 2 u -} becomes	
	\begin{equation*} 
		\begin{aligned}
			& \ \ \ \ v^{-2} (\partial_{+} \partial_{-} v + \partial_{-} \partial_{+} v) -4 v^{-3}\partial_{+} v \cdot \partial_{-} v +   \frac{v^{-2} (-u + v^{-2}\partial_{u} \mathcal{F})}{r}  ( \partial_{+} v + \partial_{-} v )  \\
			& + \frac{v^{-2} ( -v^{-1} + 2v^{-1}\mathcal{F} - \partial_{v} \mathcal{F} ) }{r}  ( \partial_{+} v - \partial_{-} v ) + \frac{v^{-2} (-2uv + 2r^{-1}uv^{-1}\partial_{u} \mathcal{F}) }{r^2} .
		\end{aligned}
	\end{equation*}		
	And the right side of \eqref{partial 2 u -} is	
	\begin{equation*}
		\begin{aligned}
			- \frac{u}{r} \left[  -v^{-2} (\partial_{+} v + \partial_{-} v ) -  \frac{2uv^{-1}}{r}   \right]  =  \frac{ uv^{-2} }{r} (\partial_{+}v + \partial_{-}v) + \frac{2u^2 v^{-1}}{r^2}  .
		\end{aligned}
	\end{equation*}	
	Then the equation \eqref{partial 2 u -} gives
	\begin{equation} \label{partial 2 + }
		\begin{aligned}
			& \ \ \ \ \partial_{+} \partial_{-} v + \partial_{-} \partial_{+} v \\
			& = 4v^{-1} \partial_{+}v \cdot \partial_{-}v + \frac{2u - v^{-2} \partial_{u} \mathcal{F}}{r} (\partial_{+}v + \partial_{-}v) \\
			& \ \ \ \ + \frac{v^{-1}-2v^{-1} \mathcal{F} + \partial_{v} \mathcal{F}}{r} (\partial_{+}v - \partial_{-}v)   + \frac{4u^2 v - 2uv^{-1}\partial_{u} \mathcal{F} }{r^{2}} .
		\end{aligned}
	\end{equation}	
	Finally, by \eqref{partial 2 - } and \eqref{partial 2 + }, we obtain the equation \eqref{character decompose prop}. The proof of this proposition is completed.
\end{proof}

By \eqref{character decompose prop}, we denote
\begin{equation} \label{partial partial v}
	\begin{aligned}
		\begin{split}
			\left\lbrace
			\begin{array}{lr}	
				\partial_{+} \partial_{-}v =2 v^{-1} \partial_{+}v \cdot \partial_{-} v + \frac{a_{11}}{2r} \partial_{-}v+ \frac{a_{12}}{2r}\partial_{+}v + \frac{a_{13}}{r^2}, \\
				\partial_{-} \partial_{+}v = 2 v^{-1} \partial_{+}v \cdot \partial_{-} v + \frac{a_{21}}{2r}\partial_{+}v+ \frac{a_{22}}{2r}\partial_{-}v + \frac{a_{23}}{r^2},
			\end{array}	
			\right.
		\end{split}	
	\end{aligned}
\end{equation}
where
\begin{equation} \label{coefficient a} 
	\begin{aligned}
		a_{11} & =3u-v^{-1}+2v^{-1} \mathcal{F}- v^{-2} \partial_{u} \mathcal{F} - \partial_{v} \mathcal{F}, \\
		a_{12} & =u+v^{-1}-2v^{-1} \mathcal{F}- v^{-2} \partial_{u} \mathcal{F} + \partial_{v} \mathcal{F}, \\
		a_{21} & =3u+v^{-1}-2v^{-1} \mathcal{F}- v^{-2} \partial_{u} \mathcal{F} + \partial_{v} \mathcal{F}, \\
		a_{22} & =u - v^{-1} + 2v^{-1} \mathcal{F}- v^{-2} \partial_{u} \mathcal{F} - \partial_{v} \mathcal{F}, \\
		a_{13}&  = a_{23} = uv(2u-v^{-2}\partial_{u} \mathcal{F}).
	\end{aligned}
\end{equation}	
Furthermore, we define
\begin{equation} 
	\begin{aligned}
		\tilde{R}_{+} \triangleq  \partial_{+} v - \frac{\alpha\mathcal{F}}{r},\ \ \tilde{R}_{-} \triangleq \partial_{-} v - \frac{\alpha\mathcal{F}}{r},
	\end{aligned}
\end{equation}
where the constants $\alpha \geq \beta \geq 1$. Then, by \eqref{partial partial v}, we have
\begin{equation} \label{eq of partial R}
	\begin{aligned}
		\begin{split}
			\left\lbrace
			\begin{array}{lr}	
				\partial_{+} \tilde{R}_{-} = 2 v^{-1}\tilde{R}_{+}\tilde{R}_{-} + \frac{\tilde{a}_{11}}{2r} \tilde{R}_{-}+ \frac{\tilde{a}_{12}}{2r}\tilde{R}_{+} + \frac{\tilde{a}_{13}}{r^2}, \\
				\partial_{-} \tilde{R}_{+} = 2 v^{-1}\tilde{R}_{+}\tilde{R}_{-} + \frac{\tilde{a}_{21}}{2r}\tilde{R}_{+}+ \frac{\tilde{a}_{22}}{2r}\tilde{R}_{-} + \frac{\tilde{a}_{23}}{r^2}, 
			\end{array}	
			\right.
		\end{split}	
	\end{aligned}
\end{equation}
where 
\begin{equation} \label{coefficient tilde a} 
	\begin{aligned}
		\tilde{a}_{12} & = 2\alpha ( 2v^{-1} \mathcal{F} + v^{-2}  \partial_{u} \mathcal{F} - \partial_{v} \mathcal{F})  + {a}_{12}, \\
		\tilde{a}_{22} & = 2\alpha (2v^{-1} \mathcal{F} - v^{-2}  \partial_{u} \mathcal{F} - \partial_{v} \mathcal{F}) + {a}_{22},\\
		\tilde{a}_{13} & = \alpha^2 \mathcal{F} \cdot ( 2v^{-1} \mathcal{F} + v^{-2}  \partial_{u} \mathcal{F} - \partial_{v} \mathcal{F} )  + \alpha [(3u + 2v^{-1}) \mathcal{F} + uv^{-1} \partial_{u} \mathcal{F} ] +a_{13}, \\
		\tilde{a}_{23} & =\alpha^2 \mathcal{F} \cdot ( 2v^{-1} \mathcal{F} - v^{-2}  \partial_{u} \mathcal{F} - \partial_{v} \mathcal{F} )  + \alpha [(3u -v^{-1}) \mathcal{F} - uv^{-1} \partial_{u} \mathcal{F} ] +a_{23}.
	\end{aligned}
\end{equation}	

The aim of introducing the variables $ \tilde{R}_{\pm} $ is to establish an `invariant region' for $\left(\partial_{+} v, \partial_{-} v \right) $. We shall use \eqref{eq of partial R} and \eqref{coefficient tilde a} and the continuous argument to prove $ \tilde{R}_{\pm} > 0 $ for $t>0$, provided the initial data satisfy the assumptions (A1)-(A2). In fact, to derive \eqref{eq of partial R} and \eqref{coefficient tilde a}, one can start from \eqref{u c pm} to get
\begin{equation*} 
	\begin{aligned}
		\partial_{+} \left(  \frac{ \mathcal{F}}{r}  \right) & = r^{-1} \partial_{+} \mathcal{F} + \mathcal{F} \partial_{+} (r^{-1}) \\
		& =  r^{-1} ( \partial_{u} \mathcal{F} \cdot \partial_{+} u + \partial_{v} \mathcal{F} \cdot \partial_{+} v ) - r^{-2} (u+v^{-1})\mathcal{F} \\
		& = r^{-1} [ \partial_{u} \mathcal{F} ( -v^{-2} \partial_{+} v - r^{-1} ( uv^{-1} + v^{-2} \mathcal{F} ) ) +  \partial_{v} \mathcal{F} \cdot \partial_{+} v] \\
		& \ \ \ \ - r^{-2} (u+v^{-1})\mathcal{F} \\
		& =  \frac{ -v^{-2} \partial_{u} \mathcal{F} + \partial_{v} \mathcal{F} }{r} \partial_{+} v + \frac{ - v^{-1} \partial_{u} \mathcal{F}  (u+v^{-1}\mathcal{F}  ) - (u+v^{-1})\mathcal{F} }{r^{2}} ,
	\end{aligned}
\end{equation*}	
and
\begin{equation*} 
	\begin{aligned}
		\partial_{-} \left(  \frac{ \mathcal{F}}{r}  \right) & = r^{-1} \partial_{-} \mathcal{F} + \mathcal{F} \partial_{-} (r^{-1}) \\
		& =  r^{-1} ( \partial_{u} \mathcal{F} \cdot \partial_{-} u + \partial_{v} \mathcal{F} \cdot \partial_{-} v ) - r^{-2} (u-v^{-1})\mathcal{F} \\
		& = r^{-1} [ \partial_{u} \mathcal{F} ( v^{-2} \partial_{-} v + r^{-1} ( uv^{-1} - v^{-2} \mathcal{F} ) ) +  \partial_{v} \mathcal{F} \cdot \partial_{-} v] \\
		& \ \ \ \ - r^{-2} (u-v^{-1})\mathcal{F} \\
		& =  \frac{ v^{-2} \partial_{u} \mathcal{F} + \partial_{v} \mathcal{F} }{r} \partial_{-} v + \frac{  v^{-1} \partial_{u} \mathcal{F}  (u-v^{-1}\mathcal{F}  ) - (u-v^{-1})\mathcal{F} }{r^{2}} .
	\end{aligned}
\end{equation*}	
Then one can obtain
\begin{equation} \label{partial R 1}
	\begin{aligned}
		\partial_{+} \tilde{R}_{-} 
		& =  \partial_{+} \partial_{-} v - \alpha \partial_{+} \left( \frac{ \mathcal{F}}{r} \right) \\
		& = 2v^{-1} \left( \tilde{R}_{+ } + \frac{\alpha \mathcal{F}}{r} \right) \cdot \left( \tilde{R}_{-} + \frac{\alpha \mathcal{F}}{r} \right)  + \frac{a_{11}}{2r} \left(\tilde{R}_{-} + \frac{\alpha \mathcal{F}}{r} \right)  \\
		& \ \ \ \  + \frac{a_{12}}{2r} \left(\tilde{R}_{+ } + \frac{\alpha \mathcal{F}}{r} \right) +\frac{a_{13}}{r^{2}} + \frac{\alpha( v^{-2} \partial_{u} \mathcal{F} - \partial_{v} \mathcal{F}  )}{r} \left( \tilde{R}_{+ } - \frac{\alpha \mathcal{F}}{r} \right) \\
		& \ \ \ \ + \frac{\alpha[v^{-1} \partial_{u} \mathcal{F}  (u+v^{-1}\mathcal{F}  ) + (u+v^{-1})\mathcal{F} ] }{r^2} \\
		& = 2v^{-1}\tilde{R}_{+ } \tilde{R}_{-} + \left( \frac{\alpha \cdot 2v^{-1} \mathcal{F} }{r} +\frac{a_{11}}{2r} \right) \tilde{R}_{-} \\
		& \ \ \ \  +\left( \frac{\alpha ( 2v^{-1}\mathcal{F} + v^{-2}\partial_{u}\mathcal{F} - \partial_{v}\mathcal{F})}{r} + \frac{a_{12}}{2r} \right)  \tilde{R}_{+ }  + \frac{\alpha^{2} \cdot 2v^{-1} \mathcal{F}^{2} }{r} \\
		& \ \ \ \ + \frac{\alpha \mathcal{F} (2u - v^{-2} \partial_{u}\mathcal{F} ) }{r} +\frac{a_{13}}{r^{2}}   + \frac{\alpha^{2} \cdot ( v^{-2} \partial_{u} \mathcal{F} - \partial_{v} \mathcal{F}  )  \mathcal{F} }{r} \\
		& \ \ \ \ + \frac{\alpha[v^{-1} \partial_{u} \mathcal{F}  (u+v^{-1}\mathcal{F}  ) + (u+v^{-1})\mathcal{F} ] }{r^2},
	\end{aligned}
\end{equation}	
and
\begin{equation} \label{partial R 2}
	\begin{aligned}
		\partial_{-} \tilde{R}_{+} & =  \partial_{-} \partial_{+} v - \alpha \partial_{-} \left( \frac{ \mathcal{F}}{r} \right) \\
		& = 2v^{-1} \left( \tilde{R}_{+ } + \frac{\alpha \mathcal{F}}{r} \right) \cdot \left( \tilde{R}_{-} + \frac{\alpha \mathcal{F}}{r} \right)  + \frac{a_{21}}{2r} \left(\tilde{R}_{+} + \frac{\alpha \mathcal{F}}{r} \right)  \\
		& \ \ \ \  + \frac{a_{22}}{2r} \left(\tilde{R}_{- } + \frac{\alpha \mathcal{F}}{r} \right) +\frac{a_{23}}{r^{2}}  - \frac{\alpha( v^{-2} \partial_{u} \mathcal{F} + \partial_{v} \mathcal{F}  )}{r} \left( \tilde{R}_{+ } - \frac{\alpha \mathcal{F}}{r} \right) \\
		& \ \ \ \ - \frac{\alpha[v^{-1} \partial_{u} \mathcal{F}  (u-v^{-1}\mathcal{F}  ) - (u-v^{-1})\mathcal{F} ] }{r^2} \\
		& = 2v^{-1}\tilde{R}_{+ } \tilde{R}_{-} + \left( \frac{\alpha \cdot 2v^{-1} \mathcal{F} }{r} +\frac{a_{21}}{2r} \right) \tilde{R}_{+} \\
		& \ \ \ \  +\left( \frac{\alpha ( 2v^{-1}\mathcal{F} - v^{-2}\partial_{u}\mathcal{F} - \partial_{v}\mathcal{F})}{r} + \frac{a_{22}}{2r} \right)  \tilde{R}_{- }  + \frac{\alpha^{2} \cdot 2v^{-1} \mathcal{F}^{2} }{r} \\
		& \ \ \ \ + \frac{\alpha \mathcal{F} (2u - v^{-2} \partial_{u}\mathcal{F} ) }{r} +\frac{a_{23}}{r^{2}}   - \frac{\alpha^{2} \cdot ( v^{-2} \partial_{u} \mathcal{F} + \partial_{v} \mathcal{F}  )  \mathcal{F} }{r}\\
		& \ \ \ \  - \frac{\alpha[v^{-1} \partial_{u} \mathcal{F}  (u-v^{-1}\mathcal{F}  ) - (u-v^{-1})\mathcal{F} ] }{r^2}.
	\end{aligned}
\end{equation}	
The equations \eqref{eq of partial R} with coefficients \eqref{coefficient tilde a} are exactly derived from \eqref{partial R 1} and \eqref{partial R 2}.

\subsection{A Priori Estimate}

According to the local existence and uniqueness of classical solution to Cauchy problem for first order quasilinear hyperbolic systems (see Chapter 1 in \cite{LiYu85}), for $C^1$ function $\bar{u}(r)$, there exists a positive constant $T > 0$ only depending on the $C^1$-norm of initial data $\bar{u}(r)$ and $v_0$, such that system \eqref{u v eq3} with initial data \eqref{ID} admits a unique $C^1$ solution on the domain 
\begin{equation} \label{def of Lambda}
	\begin{aligned}
		\Lambda (T) = \{ (r,t)|r_{+}(t,r_{1}) \leq r \leq r_{-}(t,r_{2}), \ 0\leq t \leq T  \},
	\end{aligned}
\end{equation}	
where $r=r_{+}(t,r_{1})$ and $r = r_{-}(t,r_{2})$ represent the $C_{+}$ and $C_{-}$ characteristic curves passing through the points $(r_{1},0)$ and $(r_{2},0)$, respectively. See Fig. \ref{fig1}. 

\unitlength 1mm 
\linethickness{0.4pt}
\ifx\plotpoint\undefined\newsavebox{\plotpoint}\fi 
\begin{picture}(102.86,67.3)(-15,0)
	\label{fig1}
	\put(8.11,19.05){\vector(1,0){85}}
	\put(12.61,16.05){\vector(0,1){35}}
	\multiput(13.04,39.98)(.99351,0){78}{{\rule{.4pt}{.4pt}}}
	\put(53.36,36.3){\line(0,1){0}}
	\put(10.08,50){\makebox(0,0)[cc]{$t$}}
	\put(10.08,16.97){\makebox(0,0)[cc]{$0$}}
	\put(89,16.44){\makebox(0,0)[cc]{$r$}}
	\put(20.86,16.44){\makebox(0,0)[cc]{$r_1$}}
	\qbezier(20.51,19.09)(34.82,37.3)(44.19,39.95)
	\qbezier(53.74,18.91)(75.75,37.74)(86.8,39.95)
	\qbezier(33.06,19.09)(38.36,27.93)(45.78,34.65)
	\qbezier(45.96,34.65)(43.22,29.79)(40.84,18.91)
	\put(10.25,40.13){\makebox(0,0)[cc]{$T$}}
	\put(35.36,28.28){\makebox(0,0)[cc]{$C_+$}}
	\put(25.46,32.7){\makebox(0,0)[cc]{$C_+$}}
	\multiput(36.7,19.02)(.088,0){3}{{\rule{.4pt}{.4pt}}}
	\multiput(36.88,19.02)(-.177,-.088){3}{{\rule{.4pt}{.4pt}}}
	\multiput(45.54,34.4)(-.033228,-.057817){14}{\line(0,-1){.057817}}
	\multiput(44.61,32.78)(-.033228,-.057817){14}{\line(0,-1){.057817}}
	\multiput(43.68,31.16)(-.033228,-.057817){14}{\line(0,-1){.057817}}
	\multiput(42.75,29.54)(-.033228,-.057817){14}{\line(0,-1){.057817}}
	\multiput(41.82,27.93)(-.033228,-.057817){14}{\line(0,-1){.057817}}
	\multiput(40.89,26.31)(-.033228,-.057817){14}{\line(0,-1){.057817}}
	\multiput(39.96,24.69)(-.033228,-.057817){14}{\line(0,-1){.057817}}
	\multiput(39.02,23.07)(-.033228,-.057817){14}{\line(0,-1){.057817}}
	\multiput(38.09,21.45)(-.033228,-.057817){14}{\line(0,-1){.057817}}
	\multiput(37.16,19.83)(-.033228,-.057817){14}{\line(0,-1){.057817}}
	\put(36.7,19.02){\line(0,1){0}}
	\put(47,36.42){\makebox(0,0)[cc]{$P$}}
	\put(47,26.16){\makebox(0,0)[cc]{$C_-$}}
	\put(38.5,21.74){\makebox(0,0)[cc]{$\Omega_P$}}
	\put(30.5,16.44){\makebox(0,0)[cc]{$P_+$}}
	\put(36.59,16.44){\makebox(0,0)[cc]{$P_0$}}
	\put(43,16.44){\makebox(0,0)[cc]{$P_-$}}
	\put(55.86,32){\makebox(0,0)[cc]{$\Lambda(T)$}}
	\put(72.83,27.4){\makebox(0,0)[cc]{$C_-$}}
	\put(53.21,16.44){\makebox(0,0)[cc]{$r_2$}}
	\put(21,7){Figure 3.2: Domains $\Lambda(T)$ and $\Omega_P$}
\end{picture}

Here we note that by \eqref{lambda pm} and \eqref{u c pm}, there holds
\begin{equation} 
	\begin{aligned}
		\partial_+ \lambda_- &= \partial_+ u - \partial_+ (v^{-1}) = -r^{-1}( uv^{-1}+ v^{-2}\mathcal{F}) < 0, \\
		\partial_- \lambda_+ &= \partial_- u + \partial_- (v^{-1}) = r^{-1}( uv^{-1}- v^{-2}\mathcal{F}) > 0.
	\end{aligned}
\end{equation}	
In what follows, we denote $ u_1 = \bar{u}(r_1) $ and $u_2 = \bar{u}(r_2)$. Under the assumptions (A1)-(A2), we establish the following key estimates.

\begin{lem} \label{lem u c prop}
	Suppose that system \eqref{u v eq3} with initial data \eqref{ID} admits a classical solution in $ \Lambda (T)$ for some $T>0$. Then under the assumptions (A1)-(A2), the solution satisfies 
	\begin{equation} \label{lem R}
		\begin{aligned}
			u_2 \leq u \leq u_1, \ \tilde{R}_{\pm} > 0 ,
		\end{aligned}
	\end{equation}	
	in $\Lambda (T)$, where $	\tilde{R}_{\pm} = \partial_{\pm} v - \frac{\alpha\mathcal{F}}{r}$ for $\alpha \geq 1$.
\end{lem}

\begin{proof}
	We shall prove this lemma by the argument of continuity. First, we will check the estimate \eqref{lem R} is valid for the initial data. Noting the second equation of \eqref{u v eq3}, one has
	\begin{equation*} 
		\begin{aligned}
			v_{t} = -uv_{r} - vu_r-\frac{uv}{r}.
		\end{aligned}
	\end{equation*}	
	Hence,
	\begin{equation*} 
		\begin{aligned}
			\partial_{\pm} v  = v_{t} + (u \pm v^{-1}) v_{r}  = \pm v^{-1} v_{r} - vu_r-\frac{uv}{r}.
		\end{aligned}
	\end{equation*}	
	According to (A2), we have
	\begin{equation*} 
		\begin{aligned}
			\tilde{R}_{\pm}(r,0) =  -v_0 \left(   \bar{u}^{\prime}(r) + \frac{\bar{u} + \alpha \mathcal{F}_{0}}{r} \right) >0 ,
		\end{aligned}
	\end{equation*}	
	Therefore, the estimates of \eqref{lem R} hold on $\{(r,t)|t=0, r_1 \leq r \leq r_2 \}$.

	Let $P$ be an arbitrary point in the area $\Lambda (T)$. The backward $C_{+}$ and $C_{-}$ characteristic curves are drawn from the point $P$ intersecting the $r$-axis with the points $P_{+}$ and $P_{-}$, respectively. The backward $C_{0}$ curve passing through $P$ intersects the $r$-axis at a point $P_{0}$. We denote by $\Omega_{P}\subset \Lambda (T)$ a closed triangle domain
	bounded by $\widehat{P_{+} P}$, $\widehat{P_{-} P}$	and $\widehat{P_{+} P_{-}}$, see Fig. \ref{fig1}. In the following we shall prove that if the inequalities in \eqref{lem R} hold
	for every point in $\Omega_{P} \backslash \{P\}$, then they also hold at $P$. The proof of this assertion will be given in three steps.
	
	{\textit{Step 1}}.	By the assumption that the estimates in \eqref{lem R} hold for any point in $\Omega_{P} \backslash \{P\}$, we have $\partial_{\pm}v> 0$ in $\Omega_{P} \backslash \{P\}$. Noting \eqref{u c pm}, along the characteristics $C_{+}$ and $C_{-}$, we have
	\begin{equation*} 
		\begin{aligned}
			\partial_{+} u  < 0 \ \ \text{on } \widehat{P_{+} P} \ \ \ \ \text{and} \ \ \ \	\partial_{-} u   > 0 \ \ \text{on } \widehat{P_{-} P}.
		\end{aligned}
	\end{equation*}	
	Thus we obtain 
	\begin{equation*} 
		\begin{aligned}
			u_2 < u(P_{-}) < u(P) < u(P_{+}) < u_1.
		\end{aligned}
	\end{equation*}	
	
	{\textit{Step 2}}. We shall analyze the signitures of coefficients in coupled equations \eqref{partial partial v} and \eqref{eq of partial R} in this step. By \eqref{def of F} and the direct computation, one can get
	\begin{equation} \label{partial u F}
		\begin{aligned}
			\partial_{u}\mathcal{F} & =  \frac{1}{2}\left[ -2uv^2 - \frac{2(v^2-u^{2}v^2-1) ( -2uv^2 ) - 8 uv^2 }{2\sqrt{(v^2-u^{2}v^2-1)^2 -4u^{2}v^{2} }} \right] \\
			& =  -uv^{2} \left(  1- \frac{v^2-u^{2}v^2+1 }{\sqrt{(v^2-u^{2}v^2-1)^2 -4u^{2}v^{2} }} \right) <0 ,
		\end{aligned}
	\end{equation}
	and
	\begin{equation}
		\begin{aligned}
			\partial_{v}\mathcal{F}& =  \frac{1}{2}\left[ 2v(1-u^2) - \frac{2(v^2-u^{2}v^2-1) \cdot(-2v)(1-u^2)  - 8 u^2 v }{2\sqrt{(v^2-u^{2}v^2-1)^2 -4u^{2}v^{2} }} \right]  \\
			& = v \left[ (1-u^2) - \frac{v^2(1-u^{2})^{2} - (1+u^2) }{\sqrt{(v^2-u^{2}v^2-1)^2 -4u^{2}v^{2} }} \right].
		\end{aligned}
	\end{equation}
	Hence, by \eqref{coefficient a} we have
	\begin{equation*} 
		\begin{aligned}
			a_{12} & =u+v^{-1}-2v^{-1}\mathcal{F}- v^{-2} \partial_{u}\mathcal{F} + \partial_{v}\mathcal{F} \\
			& =u+v^{-1}- v^{-1} \left[ (v^2-u^{2}v^2-1)-\sqrt{(v^2-u^{2}v^2-1)^2 -4u^{2}v^{2} }\right]  \\
			& \ \ \ \ + u \left(   1- \frac{v^2-u^{2}v^2+1 }{\sqrt{(v^2-u^{2}v^2-1)^2 -4u^{2}v^{2} }}  \right)\\
			& \ \ \ \   + v \left[ (1-u^2) - \frac{v^2(1-u^{2})^{2} - (1+u^2) }{\sqrt{(v^2-u^{2}v^2-1)^2 -4u^{2}v^{2} }} \right]   \\
			& = (u+v^{-1}) \left( 2- \sqrt{\frac{v^2-u^{2}v^2-1+2uv}{v^2-u^{2}v^2-1-2uv}} \right) ,
		\end{aligned}
	\end{equation*}	
	and
	\begin{equation*} 
		\begin{aligned}
			a_{22} & =u-v^{-1}+2v^{-1}\mathcal{F}- v^{-2} \partial_{u}\mathcal{F} - \partial_{v}\mathcal{F} \\
			& =u-v^{-1}+ v^{-1} \left[ (v^2-u^{2}v^2-1)-\sqrt{(v^2-u^{2}v^2-1)^2 -4u^{2}v^{2} }\right]  \\
			& \ \ \ \ + u \left(   1- \frac{v^2-u^{2}v^2+1 }{\sqrt{(v^2-u^{2}v^2-1)^2 -4u^{2}v^{2} }}  \right) \\
			& \ \ \ \  - v \left[ (1-u^2) - \frac{v^2(1-u^{2})^{2} - (1+u^2) }{\sqrt{(v^2-u^{2}v^2-1)^2 -4u^{2}v^{2} }} \right]   \\
			& = (u-v^{-1}) \left( 2- \sqrt{\frac{v^2-u^{2}v^2-1-2uv}{v^2-u^{2}v^2-1+2uv}} \right),
		\end{aligned}
	\end{equation*}	
	where we use the assumption (A1). Obviously, by \eqref{partial u F}, we also have
	\begin{equation*} 
		\begin{aligned}
			a_{13} = a_{23} = uv(2u-v^{-2}\partial_{u}\mathcal{F})> 0.
		\end{aligned}
	\end{equation*}	
	Noting \eqref{coefficient tilde a}, one has
	\begin{equation} 
		\begin{aligned}
			\tilde{a}_{12} & = 2\alpha (2v^{-1}\mathcal{F}+ v^{-2} \partial_{u}\mathcal{F} - \partial_{v}\mathcal{F}) + {a}_{12} \\
			& = 2\alpha (- a_{12} +u +v^{-1}) + {a}_{12}\\
			& = (u +v^{-1}) \left[  (2 \alpha -1) \sqrt{\frac{v^2-u^{2}v^2-1+2uv}{v^2-u^{2}v^2-1-2uv}} - 2 \alpha+2  \right] > 0,
		\end{aligned}
	\end{equation}	
	for any $ \alpha \geq 1$, and
	\begin{equation}
		\begin{aligned}
			\tilde{a}_{22} & = 2\alpha (2v^{-1}\mathcal{F}- v^{-2} \partial_{u}\mathcal{F} - \partial_{v}\mathcal{F}) + {a}_{22} \\
			& = 2\alpha (a_{22} -(u -v^{-1})) + {a}_{22}\\
			& = (u -v^{-1}) \cdot  2 \alpha \left( -  \sqrt{\frac{v^2-u^{2}v^2-1-2uv}{v^2-u^{2}v^2-1+2uv}} +1 \right) \\
			& \ \ \ \ + (u -v^{-1}) \left(  2 - \sqrt{\frac{v^2-u^{2}v^2-1-2uv}{v^2-u^{2}v^2-1+2uv}} \right) > 0,
		\end{aligned}
	\end{equation}	
	for any $ \alpha \geq 0$. Besides, when $ \alpha =1$, one has
	\begin{equation*} 
		\begin{aligned}
			\tilde{a}_{13} &= ( 3u + v^{-1} + 2v^{-1} \mathcal{F} + v^{-2}  \partial_{u} \mathcal{F} - \partial_{v} \mathcal{F} ) \mathcal{F} + 2u^2 v \\
			& = (- a_{12} +4u +2v^{-1})+ 2u^2 v \\
			& = \left[ 2u + (u + v^{-1}) \sqrt{\frac{v^2-u^{2}v^2-1+2uv}{v^2-u^{2}v^2-1-2uv}} \right] \mathcal{F} + 2u^2 v  > 0,
		\end{aligned}
	\end{equation*}	
	and
	\begin{equation*} 
		\begin{aligned}
			\tilde{a}_{23} &= ( 3u + v^{-1} - v^{-2}  \partial_{u} \mathcal{F} - \partial_{v} \mathcal{F} ) \mathcal{F} + 2uv (u - v^{-2}  \partial_{u} \mathcal{F}) \\
			& = \left[ a_{22} + 2u + v^{-1} \left( 3- v^{2} (1-u^{2}) + \sqrt{\frac{v^2-u^{2}v^2-1-2uv}{v^2-u^{2}v^2-1+2uv}} \right) \right] \mathcal{F} \\
			& \ \ \ \ + 2uv (u - v^{-2}  \partial_{u} \mathcal{F})  > 0.
		\end{aligned}
	\end{equation*}	
	Moreover, it is clear that 
	\begin{equation} \label{tilde a13 a23} 
		\begin{aligned}
			\tilde{a}_{13}>0, \ \tilde{a}_{23}>0,
		\end{aligned}
	\end{equation}	
	when $ \alpha > 1$.

	{\textit{Step 3}}.	Now we prove $\tilde{R}_{+} (P) > 0$. Suppose $\tilde{R}_{+}(P) = 0$ and $\tilde{R}_{-}(P) \geq 0$. Then by the assumption that the inequalities in \eqref{lem R} hold for every point in $\Omega_{P} \backslash \{P\}$, we have $\partial_{-}\tilde{R}_{+} \leq 0$ at $P$. However, by the first equation of \eqref{eq of partial R} and \eqref{tilde a13 a23} at the point $P$, we derive
	\begin{equation*} 
		\begin{aligned}
			\partial_{-} \tilde{R}_{+} & = \frac{\tilde{a}_{12}}{2r} \tilde{R}_{-} + \frac{\tilde{a}_{13}}{r^2} > 0 ,
		\end{aligned}
	\end{equation*}	
	which leads to a contradiction. Hence, we have
	$\tilde{R}_{+}(P) > 0$. Similarly, one can get $\tilde{R}_{-}(P)> 0$. Therefore we finish the proof of the
	assertion. 
	
	Consequently, we obtain the estimate \eqref{lem R}. We complete the proof of this lemma.
	
\end{proof}

\subsection{Proof of Theorem \ref{main thm}}


Now we give the proof of Theorem \ref{main thm}.	Let the characteristics $C_{0}^{\eta_1}: r = r_{0}(t;\eta_1)$ and $C_{0}^{\eta_2}: r = r_{0}(t;\eta_2)$ be determined by
	\begin{equation*} 
		\begin{aligned}
			\begin{split}
				\left\lbrace
				\begin{array}{lr}	
					\frac{\mathrm{d}r_{0}(t;\eta_1)}{\mathrm{d}t} = u(r_{0}(t;\eta_1),t)	, \ t>0, \\
					r_{0}(0;\eta_1) =\eta_1	,			 
				\end{array}	
				\right.
			\end{split}	
		\end{aligned}
	\end{equation*}
	and
	\begin{equation*} 
		\begin{aligned}
			\begin{split}
				\left\lbrace
				\begin{array}{lr}	
					\frac{\mathrm{d}r_{0}(t;\eta_2)}{\mathrm{d}t} = u(r_{0}(t;\eta_2),t)	, \ t>0, \\
					r_{0}(0;\eta_2) =\eta_2,				 
				\end{array}	
				\right.
			\end{split}	
		\end{aligned}
	\end{equation*}
	respectively. By the conservation law of the second equation in \eqref{u v eq3}, we have
	\begin{equation}
		\begin{aligned}
			\int_{r_{0}(0;\eta_1)}^{r_{0}(0;\eta_2)} [ (rv)_t + (ruv)_r ] \mathrm{d} r = 0,
		\end{aligned}
	\end{equation}	
	and thus
	\begin{equation} \label{conservation of mass}
		\begin{aligned}
			\frac{\mathrm{d}}{\mathrm{d}t} \int_{r_{0}(0;\eta_1)}^{r_{0}(0;\eta_2)} r v (r,t) \mathrm{d}r = 0.
		\end{aligned}
	\end{equation}	
	By virtue of \eqref{u c pm} and \eqref{lem R}, one has
	\begin{equation} 
		\begin{aligned}
			\partial_{0} u &= - \frac{v^{-2}}{2}(\partial_{+} v - \partial_{-} v) - \frac{v^{-2} \mathcal{F}}{r}  \\
			& = \partial_{0} (v^{-1}) + v^{-2}\left( \partial_{-} v - \frac{\mathcal{F}}{r} \right) \\
			& > \partial_{0} (v^{-1}) + v^{-2} \tilde{R}_- \\
			& > \partial_{0} (v^{-1}),
		\end{aligned}
	\end{equation}	
	where $\partial_{0}  = \partial_{t} + u \partial_{r} = \frac{1}{2} (\partial_{+} + \partial_{-} )$. Integrating the above inequality along the characteristic $C_{0}^{\eta_1}$ and noting \eqref{lem R}, we immediately have
	\begin{equation} \label{3.12}
		\begin{aligned}
			u > \bar{u} (\eta_1) - v_{0}^{-1} \ \ \ \ \text{on } C_{0}^{\eta_1}.
		\end{aligned}
	\end{equation}	
	On the other hand, one can get
	\begin{equation} 
		\begin{aligned}
			\partial_{0} u & = - v^{-2}(\partial_{+} v - \partial_{0} v) -\frac{ v^{-2}\mathcal{F}}{r} \\
			& = -\partial_{0} (v^{-1}) - v^{-2}\left( \partial_{+} v - \frac{\mathcal{F}}{r} \right) \\
			& < - \partial_{0} (v^{-1}) - v^{-2} \tilde{R}_- \\
			& < - \partial_{0} (v^{-1}).
		\end{aligned}
	\end{equation}	
	Integrating the above inequality along the characteristic $C_{0}^{\eta_2}$ and noting \eqref{lem R}, we have
	\begin{equation} \label{3.14}
		\begin{aligned}
			u < \bar{u} (\eta_2) + v_{0}^{-1} \ \ \ \ \text{on } C_{0}^{\eta_2}.
		\end{aligned}
	\end{equation}	
	Noting \eqref{u c pm} and \eqref{lem R}, we have $ \partial_{+}(u+v^{-1}) < 0 $ and $ \partial_{-}(u-v^{-1}) > 0 $. Thus, on the characteristic curve $r=r_{+}(t)$, we get
	\begin{equation} \label{3.15 1}
		\begin{aligned}
			u+v^{-1}<u_1+v^{-1}_0 .		 
		\end{aligned}
	\end{equation}
	Similarly, on the characteristics $r=r_{-}(t)$, one can get
	\begin{equation} \label{3.15 2}
		\begin{aligned}
			u-v^{-1}>u_2-v^{-1}_0.			 
		\end{aligned}
	\end{equation}
	Combining \eqref{3.12}, \eqref{3.14}, \eqref{3.15 1} and \eqref{3.15 2}, we obtain
	\begin{equation} 
		\begin{aligned}
			& r_{-}(t) - r_{0}(t;\eta_2) > r_{2}-\eta_2-( \bar{u} (\eta_2) - u_2 + 2v^{-1}_{0} ) t, \\
			& r_{0}(t;\eta_1)-r_{+}(t) > \eta_1-r_{1}-( u_1 - \bar{u} (\eta_1)  + 2v^{-1}_{0} ) t,
		\end{aligned}
	\end{equation}	
	which yields that
	\begin{equation} 
		\begin{aligned}
			r_{0}(t;\eta_2) - r_{0}(t;\eta_1) < \eta_2-\eta_1 - (\bar{u} (\eta_1) - \bar{u} (\eta_2)-2v^{-1}_{0})t .
		\end{aligned}
	\end{equation}	
	Hence, it follows from (A3) that if system \eqref{u v eq3} with initial data \eqref{ID} admits a classical solution, then there exists a $0<t^{*}<\frac{\eta_2-\eta_1 }{\bar{u} (\eta_1) - \bar{u} (\eta_2)-2v^{-1}_{0}}$ such that 
	\begin{equation*} 
		\begin{aligned}
			r_{+}(t^{*}) < r_{0}(t^{*};\eta_2) = r_{0}(t^{*};\eta_1) < r_{-}(t^{*}),
		\end{aligned}
	\end{equation*}	
	and
	\begin{equation*} 
		\begin{aligned}
			r_{+}(t) < r_{0}(t;\eta_2) < r_{0}(t;\eta_1) < r_{-}(t) \ \ \ \ \text{for } t<t^{*}.
		\end{aligned}
	\end{equation*}	
	Consequently, by \eqref{conservation of mass} we obtain that there exists a $t_{*} \in (0, t^{*}]$ such that system \eqref{u v eq3} with initial data \eqref{ID} admits a classical solution in $\Lambda(t_{*})$ and 
	$$ 	\lim\limits_{t \rightarrow t_{*}^{-}} \left\| v (\cdot,t) \right\|_{0;\Lambda(t)} = + \infty. $$
	This completes the proof of Theorem \ref{main thm}.

We end up this article with an application of Theorem \ref{main thm} showing the singularity formation of the relativistic membrane equation. Recall that 
\begin{equation}
	\begin{aligned}
		u=-\frac{\phi_{r}\phi_{t}}{1+\phi_{r}^{2}}, \ 
		v=\frac{1+\phi_{r}^{2}}{\sqrt{\Delta}},
	\end{aligned}
\end{equation}
and $ \Delta = 1+\phi_{r}^{2}-\phi_{t}^{2} $. The vanishing of $\Delta$ causes the blow-up of $v$ since $\phi_{r}$ is assumed to be regular in the domain $\Lambda(t_*)$. This
 precisely characterizes the nature of singularities. Indeed, the singularity occurs when the hypersurface turns from being timelike to being null. 

\section*{Acknowledgement}

Lv Cai is partially supported by Shanghai Post-doctoral Excellence Program (Grant No.2024223). Jianli Liu is partially supported by Natural Science Foundation of Shanghai Municipality (Grant No.25ZR1401123).

$\,$

$\,$

\end{document}